\documentclass[11pt,oneside,reqno]{amsart}
\usepackage{geometry}
\usepackage[utf8]{inputenc}
\usepackage{amstext,latexsym,amsbsy,amsmath,amssymb,amsthm,mathtools,relsize,geometry,enumerate}
\usepackage{hyperref}
\hypersetup{
  colorlinks   = true, 
  urlcolor     = blue, 
  linkcolor    = red, 
  citecolor   = red 
}
\usepackage{mathtools,enumerate,enumitem}

\usepackage{graphicx}
\usepackage{epstopdf}
\setkeys{Gin}{width=\linewidth,totalheight=\textheight,keepaspectratio}
\graphicspath{{./images/}}

\usepackage{multicol}
\usepackage{booktabs}
\usepackage{mathrsfs}
\usepackage{color}
\usepackage{enumitem}

\newcommand{\nbb}{\mathbb{N}}
\newcommand{\rbb}{\mathbb{R}}

\renewcommand{\L}{\mathcal{L}}

\newcommand{\W}{\mathcal{W}}

\newcommand{\B}{\mathcal{B}}

\newcommand{\nt}{\notag}

\newcommand{\x}{\mathrm{x}}
\newcommand{\y}{\mathrm{y}}
\newcommand{\z}{\mathrm{z}}

\newcommand{\Xbf}{\mathbf{X}}

\newcommand{\R}{\mathcal{R}}

\renewcommand{\d}{\textup{d}}

\newcommand{\f}{\varphi}

\newcommand{\E}{\mathbb{E}}

\renewcommand{\P}{\mathbb{P}}

\newcommand{\M}{\mathcal{M}}

\theoremstyle{plain}
\newtheorem{theorem}{Theorem}[section]

\newtheorem{lemma}[theorem]{Lemma}
\newtheorem{assumption}[theorem]{Assumption}
\newtheorem{proposition}[theorem]{Proposition}

\theoremstyle{definition}
\newtheorem{definition}[theorem]{Definition}

\newtheorem{remark}[theorem]{Remark}

\numberwithin{equation}{section}

\title{Geometric ergodicity of a stochastic Hamiltonian system}

\author{ Hung D.~Nguyen$^1$ and Lekun Wang$^2$}

\thanks{\noindent \hspace{-0.52cm} $^1$ Department of Mathematics, University of Tennessee, Knoxville, Tennessee, USA}
\thanks{\hspace{-0.25cm} Email: hnguye53@utk.edu}
\thanks{\noindent $^2$ Master's Program in Computational and Applied Mathematics, The University of Chicago, Chicago, Illinois, USA}
\thanks{\hspace{-0.25cm} Email: lekunbillwang@uchicago.edu}

\begin{document}

\maketitle

\begin{abstract}
We study the long time statistics of a two-dimensional Hamiltonian system in the presence of Gaussian white noise. While the original dynamics is known to exhibit finite time explosion, we demonstrate that under the impact of the stochastic forcing as well as a deterministic perturbation, the solutions are exponentially attractive toward the unique invariant probability measure. This extends previously established results in which the system is shown to be noise-induced stable in the sense that the solutions are bounded in probability.
\end{abstract}

\medskip


\section{Introduction} \label{sec:intro}
There are many deterministic systems whose solutions only exist up to a finite time window. Interestingly, by adding a suitable stochastic forcing, it can be shown that these dynamics become non-explosive and even further stabilize as time tends to infinity \cite{athreya2012propagating,bodova2012noise,cerrai2005stabilization,
herzog2015noiseI,herzog2015noiseII,herzog2021stability,
kolba2019noise,leimbach2020noise,leimbach2017blow,
scheutzow1993stabilization}. This phenomenon is typically known as either \emph{ noise-induced stability} or \emph{noise-induced stabilization} \cite{athreya2012propagating,leimbach2020noise}. Results in this direction appeared as early as in the work of \cite{scheutzow1993stabilization,scheutzow1995integral} in which noise stabilization is dimensional sensitive for a class of SDEs. More specifically, while in dimension 1, this system exhibits finite time blow-up, in dimension two, it is relaxing toward the unique invariant probability measure exponentially fast \cite{athreya2012propagating,herzog2015noiseI,herzog2015noiseII}. Similar results concerning stationary distributions are central in the work of  \cite{gawedzki2011ergodic,herzog2021stability} for a model arising in turbulent transport. Analogous results for a reaction-diffusion equation was previously established in \cite{cerrai2005stabilization}. Another large-time behavior of interest is the existence of global random attractors which are investigated in \cite{leimbach2020noise,leimbach2017blow}. 

In the present article, we are interested in the statistically steady states of the following system in $\rbb^2$:
\begin{align}\label{eq:system}
        \d{}x_t& = \Big(h'(x^m_ty^n_t)x^{m-1}_ty^{n-1}_t -\big|h'(x^m_t y^n_t )x^{m-1}_t y^{n-1}_t \big|^q\Big)n x_t  \d t + \epsilon_x \d{}B_t^1, \nt \\ 
 \d{}y_t& = \Big(- h'(x^m_t y^n_t )x^{m-1}_t y^{n-1}_t  -\big|h'(x^m_t y^n_t )x^{m-1}_t y^{n-1}_t \big|^q\Big) my_t \d t  + \epsilon_y \d{}B_t^2.
\end{align} 
In the above, $h:\rbb\to\rbb$ represents the Hamiltonian satisfying
\begin{align*}
\inf_{x\in\rbb}|h'(x)|>0,
\end{align*}
$(B^1_t,B^2_t)$ is a two-dimensional standard Brownian motion, {$m\ge2,n\ge2$ are integers}, and $q\in (1,\infty),\,\epsilon_x,\,\epsilon_y$ are positive constants. We note that \eqref{eq:system} is regarded as a {type of perturbation} of a deterministic Hamiltonian dynamics of the form
\begin{align}
\frac{\d}{\d t}x_t&= \partial_y H(x_t,y_t), \nt \\
\frac{\d}{\d t}y_t&= -\partial_x H(x_t,y_t), \label{eq:system:Hamiltonian}
\end{align}  
where
\begin{align*}
H(x,y) = h(x^my^n).
\end{align*}
Concerning \eqref{eq:system:Hamiltonian}, notably, in the case $m=n$, the well-posedness is guaranteed whereas in the case $m\neq n$, the solution exhibits finite-time explosion \cite{kolba2019noise}. Indeed, following the discussion in \cite[Section 2]{kolba2019noise}, when $m=n$, the global solution of \eqref{eq:system:Hamiltonian} is given by
\begin{align*}
x_t & = x_0\exp\Big\{ h'(x_0^my_0^m)mx_0^{m-1}y_0^{m-1} t \Big\},\\
y_t & = y_0\exp\Big\{ -h'(x_0^my_0^m)mx_0^{m-1}y_0^{m-1} t \Big\}.
\end{align*}
On the other hand, when $m\neq n$, using the ``smart" change of variable
\begin{align*} 
    z=x^{m-1}y^{n-1},
\end{align*}
together with the fact that \eqref{eq:system:Hamiltonian} satisfies
\begin{align*}
    h(x_t^my_t^n) = h(x_0^my_0^n),
\end{align*}
observe that
\begin{align} \label{eq:z_t:Hamiltonian}
    \frac{\d}{\d t}z_t = h'(x_t^my_t^n)(m-n)x_t^{2m-2}y_t^{2n-2 } = h'(x_0^my_0^n)(m-n)z_t^2.
\end{align}
Then, we can explicitly solve for $z_t$ and ultimately derive the formula
\begin{align*}
x_t &= x_0\big(1-h'(x_0^my_0^n)(m-n)x_0^{m-1}y_0^{n-1} \cdot t\big)^{\frac{n}{n-m}},\\
y_t & = y_0\big(1-h'(x_0^my_0^n)(m-n)x_0^{m-1}y_0^{n-1} \cdot t\big)^{\frac{m}{m-n}}.
\end{align*}
In particular, when $m\neq n$, system \eqref{eq:system:Hamiltonian} does not possess global solutions for all time. To overcome the well-posedness issue, as a first ansatz, we may resort to the approach of adding white noise and consider the stochastic analogue of \eqref{eq:system:Hamiltonian} given by
\begin{align}\label{eq:system:Hamiltonian:noise}
        \d{}x_t& = h'(x^m_ty^n_t)nx^{m}_ty^{n-1}_t   \d t + \epsilon_x \d{}B_t^1, \nt \\ 
 \d{}y_t& = - h'(x^m_t y^n_t )mx^{m-1}_t y^{n}_t   \d t  + \epsilon_y \d{}B_t^2.
\end{align} 
However, as it turns out, when $m=n$, \eqref{eq:system:Hamiltonian:noise} does not stabilize for large-time $t$ \cite{kolba2019noise} whereas when $m\neq n$, the solution does not even exist globally. This stems from the fact that there is still a lack of strong dissipation in \eqref{eq:system:Hamiltonian:noise}, allowing for the solution to grow indefinitely out of control. In \cite{kolba2019noise}, the authors circumvent the instability by adding a deterministic perturbation to \eqref{eq:system:Hamiltonian:noise} as follow:
 \begin{align}\label{eq:system:q=2}
        \d{}x_t& = h'(x^m_ty^n_t)n x^{m}_ty^{n-1}_t\d t -\big|h'(x^m_t y^n_t )\big|^2 nx^{2m-1}_t y^{2n-2}_t   \d t + \epsilon_x \d{}B_t^1, \nt \\ 
 \d{}y_t& = - h'(x^m_t y^n_t )m x^{m-1}_t y^{n}_t\d t  -\big|h'(x^m_t y^n_t )\big|^2 mx^{2m-2}_t y^{2n-1}_t  \d t  + \epsilon_y \d{}B_t^2.
\end{align}
We note that \eqref{eq:system:q=2} becomes a special case of \eqref{eq:system} when $q=2$ in \eqref{eq:system}. The motivation of \eqref{eq:system:q=2} can be explained as follows: when $\epsilon_x=\epsilon_y=0$ in \eqref{eq:system:q=2}, $z_t=x_t^{m-1}y_t^{n-1}$ satisfies the deterministic equation
\begin{align} \label{eq:z_t:Hamiltonian:perturb}
    \frac{\d}{\d t}z_t = h'(x_t^my_t^n)(m-n)z_t^2 - |h'(x_t^my_t^n)|^2(m(n-1)+n(m-1))z_t^3.
\end{align}
In comparison with \eqref{eq:z_t:Hamiltonian}, the appearance of the second term on the right-hand side of \eqref{eq:z_t:Hamiltonian:perturb} formally plays the role of a potential, pushing $z_t$ back to the origin whenever its values gets too large. Thus, while \eqref{eq:system:q=2} breaks the Halmitonian structure, it creates an energy loss effect, which is otherwise not available in either \eqref{eq:system:Hamiltonian} or \eqref{eq:system:Hamiltonian:noise}. We also would like to remark that this modification of the original equation \eqref{eq:system:Hamiltonian} is purely for mathematical purpose. The issue of exploring other interesting forms of stabilization for \eqref{eq:system:Hamiltonian} is left open for future projects.

Following the framework of \cite{khasminskii2011stochastic}, it can be shown that \eqref{eq:system:q=2} is stable in the sense that the solution is bounded in probability \cite[Theorem 1]{kolba2019noise}, cf. \eqref{def:stable} below. The stability argument relies on the construction of a suitable Lyapunov function proving that the dynamics is always recurrent in finite time. As a byproduct, \eqref{eq:system:q=2} admits at least one invariant probability measure. The uniqueness of an invariant measure for \eqref{eq:system:q=2} however remained unresolved in \cite{kolba2019noise}.

Our main goal in this note is thus to address the question of unique ergodicity and ultimately the mixing rate for the general system \eqref{eq:system} (for all $q>1$). More specifically, we will show that \eqref{eq:system} is always exponentially attractive toward the unique invariant probability measure, cf. Theorem \ref{thm:ergodicity} below. Following the framework of \cite{hairer2011yet,meyn2012markov}, the proof of Theorem \ref{thm:ergodicity} relies on two important ingredients: a minorization condition proving that one may couple a solution pair once both have arrived in
the center of the phase space, and a suitable Lyapunov function showing the dynamics return to the center exponentially fast. While the minorization condition is quite standard making use of the ellipticity of the system, that is noise is present in all directions, the construction of Lyapunov functions is more delicate that requires a deeper understanding of the behavior of the solutions. To this end, we draw upon the method employed in \cite{kolba2019noise} dealing with the same large-time issue for the particular equation \eqref{eq:system:q=2}. Yet, while the arguments in \cite{kolba2019noise} are sufficient to establish stochastical boundedness, they do not produce the quantitative control needed to obtain unique ergodicity. In this work, we tackle the issue by refining the proofs in \cite{kolba2019noise} tailored to  the general system \eqref{eq:system} and deriving stronger dissipative estimates, which are very convenient for the purpose of establishing a convergence rate. The technique of Lyapunov employed here is also motivated by previous work in \cite{athreya2012propagating,
herzog2015noiseI,herzog2015noiseII, herzog2017ergodicity,lu2019geometric}. Namely, we perform a heuristic asymptotic scaling to determine the leading order terms in the ``bad" regions where the corresponding deterministic Hamiltonian system \eqref{eq:system:Hamiltonian} exhibits finite time explosions, namely, when $|x|$ is small whereas $|y|$ is large, and vice versa (see Figure \ref{fig:simulations}). It is worthwhile to point out that noise in both directions will be facilitated to drive the stochastic flow away from these regions. On the other hand, using the same scaling shows that when $|x|$ and $|y|$ are both large, the $q$-terms in \eqref{eq:system} are the dominating quantities, forcing the dynamics returning exponentially fast. We emphasize that the general structures of our Lyapunov functions are not new as they were discussed in \cite{kolba2019noise}, see also Remark \ref{rem:kolba2}. Nevertheless, the heuristic scaling is particularly helpful as to explain why the presence of noise and the $q$-term is crucial in \eqref{eq:system}. This argument will be presented in Section \ref{sec:lyapunov:intuition} whereas the main proofs of the Lyapunov functions are supplied in Section \ref{sec:lyapunov:R1R2R3} and Section \ref{sec:lyapunov:gluing}.

The rest of the paper is organized as follows. In Section~\ref{sec:result}, we introduce the relevant notations and the main assumptions. We also state the main results of the paper, including Proposition~\ref{prop:well-posed} giving the well-posedness and Theorem \ref{thm:ergodicity}, establishing unique ergodicity as well as an exponential mixing rate. In Section~\ref{sec:lyapunov}, we discuss the Lyapunov construction including a heuristic argument to build up intuition about the dynamics \eqref{eq:system}. We also provide the rigorous proofs of Lyapunov functions in this section. In Section \ref{sec:proof-of-main-theorem}, we prove the main ergodicity result by making use of a minorization condition as well as the estimates collected in Section~\ref{sec:lyapunov}.

\section{Assumptions and main results} \label{sec:result}

Throughout, we let $(\Omega, \mathcal{F}, (\mathcal{F}_t)_{t\geq 0},  \P)$ be a filtered probability space satisfying the usual conditions \cite{karatzas2012brownian} and $B_t^1,\,B_t^2$ are i.i.d. standard Brownian motions adapted to the filtration $(\mathcal{F}_t)_{t\ge 0}$.

Concerning the nonlinearity, we will make the following assumption throughout the paper.
\begin{assumption} \label{cond:h}
The function $h\in C^1(\rbb)$ satisfies \begin{equation}\label{cond:h:h'>a}
|h'(t)|>a,\quad t\in\rbb,
\end{equation}
for some positive constant $a$. In other words, $h'$ is bounded away from zero and either strictly positive or strictly negative.
\end{assumption}
With regard to the constants $q,m,n$ in \eqref{eq:system}, we assume that they are bounded from below by 1.
\begin{assumption} \label{cond:m_and_n} The constants $q$, $m$ and $n$ satisfy $m,n\in\nbb$ and $q, m, n \,\in (1,\infty)$.
\end{assumption}

\begin{remark} \label{rem:qmn}
We note that the condition $m,n>1$ is nominal ensuring that the Hamiltonian $H(x,y)=h(x^my^n)$ is differentiable, so as to avoid singularity in the drift terms of \eqref{eq:system}. On the other hand, the condition $q>1$ is to create a dissipative effect dominating the Hamiltonian,  especially when $|x|$ and $|y|$ are both large. See the proof of Lemma \ref{lem:v1} below. In turn, this allows for statistical equilibrium to be
reached with an exponential convergence rate.  
\end{remark}
Under the above assumptions, it can be shown that system \eqref{eq:system} is always well-posed. More precisely, we have the following result.
\begin{proposition} \label{prop:well-posed}
For every initial condition $\x_0=(x_0,y_0)\in\rbb^2$, system \eqref{eq:system} admits a unique strong solution $\x_t=(x_t,y_t)\in\rbb^2$.
\end{proposition}
Following the approach in \cite{kolba2019noise}, the proof of Proposition \ref{prop:well-posed} is a consequence of the existence of suitable Lyapunov functions, cf. Section \ref{sec:lyapunov}. 
As a consequence of the well-posedness, we can thus introduce the Markov transition probabilities of the solution $\x_t$ by
\begin{align*}
P_t(\x_0,A):=\P_{\x_0}(\x_t\in A),
\end{align*}
which are well-defined for $t\ge 0$, initial condition $\x_0\in\rbb^2$ and Borel sets $A\subset \rbb^2$. Letting $\B_b(\rbb^2)$ denote the set of bounded Borel measurable functions $f:\rbb^2 \rightarrow \rbb$, the associated Markov semigroup $P_t:\B_b(\rbb^2)\to\B_b(\rbb^2)$ is defined and denoted by
\begin{align*}
P_t f(\x_0)=\E_{\x_0}[f(\x_t) ], \,\, f\in \B_b(\rbb^2).
\end{align*}
Recall that a probability measure $\mu$ on Borel subsets of $\rbb^2$ is called {\bf invariant} for the semigroup $P_t$ if for every $f\in \B_b(\rbb^2)$
\begin{align*}
\int_{\rbb^2} f(x,y) P_t^*\mu(\d x,\d y)=\int_{\rbb^2} f(x,y)\mu(\d x,\d y),
\end{align*}
where $P_t^*\mu$ denotes the push-forward measure of $\mu$ by $P_t$, i.e.,
\begin{align*}
\int_{\rbb^2}f(x,y)P_t^*\mu(\d x,\d y) = \int_{\rbb^2} P_t f(x,y)\mu (\d x,\d y).
\end{align*}
Next, we let $\mathcal{L}$ be the generator  associated with \eqref{eq:system} and given by
\begin{equation*}
    (\mathcal{L}\f )(x,y)=\lim_{t\to 0}\frac{\mathbb{E}_{(x,y)}[\f (X_t)]-\f(x,y)}{t}
\end{equation*}
In particular, one defines $\mathcal{L}$ for any $v\in C^2(\rbb^2;\rbb)$ by the following expression
\begin{equation}\label{form:L}
     \mathcal{L}v= (w-w^q)nx\frac{\partial}{\partial x}v+\frac{\epsilon_x^2}{2}\frac{\partial^2}{\partial x^2}v+(-w-w^q)my \frac{\partial}{\partial y}v+\frac{\epsilon_y^2}{2}\frac{\partial^2}{\partial y^2}v,
\end{equation} 
where 
\begin{align}
  w(x,y)&:=h'(x^my^n)x^{m-1}y^{n-1}.\label{eq:w}
\end{align}
In \cite{kolba2019noise}, it was shown that there exists a suitable function $V\in C^2(\rbb^2;\rbb)$ such that
\begin{align} \label{ineq:LV->-infty}
\L V \to -\infty,\quad\text{as }\, |(x,y)|\to \infty,
\end{align}
As a consequence, the process $(x_t,v_t)$ is globally stable. That is for all initial condition $(x_0,y_0)$ and for all $\delta\in(0,1)$, there exists a positive constant $M=M(\delta,x_0,y_0)$ sufficiently large such that
\begin{align} \label{def:stable}
\P(|(x_t,y_t)|\le M)>1-\delta,\quad t\ge 0.
\end{align}
Furthermore, it can be shown that the following sequence of average measures 
\begin{align*}
\mu_T(\cdot)=\frac{1}{T}\int_0^T P_t((0,0),\cdot)\d t,
\end{align*}
is tight. By virtue of the Krylov-Bogoliubov procedure, up to a subsequence, $\mu_T$ converges weakly to a limiting measure $\mu_\infty$, which is invariant for \eqref{eq:system}.

With regard to the uniqueness of $\mu_\infty$ as well as the convergent rate of $P_t$ toward $\mu_\infty$, we will work with suitable Wasserstein distances. Following the framework of \cite{hairer2011yet,meyn2012markov}, for a measurable function $V:\rbb^2\to(0,\infty)$, we introduce the weighted supremum norm  defined as
\begin{align*}
\|\f\|_V:=\sup_{(x,y)\in\rbb^2}\frac{|f(x,y)|}{1+V(x,y)}.
\end{align*} 
We denote by $\M_V$ the collection of probability measures $\mu$ on Borel subsets of $\rbb^2$ such that
\begin{align*}
\int_{\rbb^2}V(x,y)\mu(\d x,\d y)<\infty.
\end{align*}
Let $\W_V$ be the corresponding Wasserstein distance in $\M_V$ associated with $\|\cdot\|_V$, given by
\begin{align*}
\W_V(\mu_1,\mu_2)=\sup_{\|\f\|_V\le 1}\Big|\int_{\rbb^2} \f(x,y)\mu_1(\d x,\d y)-\int_{\rbb^2} \f(x,y)\mu_2(\d x,\d y)\Big|.
\end{align*}
We refer the reader to the monograph \cite{villani2021topics} for a detailed account of Wasserstein distances and optimal transport problems. With this setup, we can now state the main result of the paper:

\begin{theorem} \label{thm:ergodicity}
1. The Markov semigroup $P_t$ admits a unique invariant probability measure $\pi$.

2. There exists a function $V\in C^2(\rbb^2;[1,\infty))$ such that for all $\mu\in \M_V$, the following estimate holds 
\begin{align} \label{ineq:exponential-rate}
\W_V\big( P_t^*\mu,\pi \big) \le Ce^{-ct}\W_V(\mu,\pi),\quad t\ge 0,
\end{align}
for some positive constants $C$ and $c$ independent of $\mu$ and $t$. 
\end{theorem}

The proof of Theorem \ref{thm:ergodicity} will follow by establishing the existence of appropriate Lyapunov functions, cf. Definition \ref{def:lyapunov}, as well as a minorization condition, cf. Definition \ref{def:minorization}. The former shows that the process returns exponentially fast to a bounded set around the origin whereas the latter is needed to couple a solution pair once both have arrived at the center. While the minorization is quite standard and follows the classical Stroock-Varadhan support theorem, the Lyapunov construction is more delicate requiring a deeper understanding of the dynamics. In particular, we will employ the functions found in \cite{kolba2019noise} and prove that they satisfy dissipative estimates stronger than \eqref{ineq:LV->-infty}. This will be explained in details in Section \ref{sec:lyapunov}. The minorization as well as the proof of Theorem \ref{thm:ergodicity} will be supplied in Section \ref{sec:proof-of-main-theorem}.

\section{Lyapunov function} \label{sec:lyapunov}

Throughout the rest of the paper, $c$ and $C$ denote generic positive constants that may change from line to line. The main parameters that they depend on will appear between parenthesis, e.g., $c(T,q)$ is a function of $T$ and $q$. 

In this section, we draw upon the Lyapunov approach in \cite{kolba2019noise} and explicitly construct a Lyapunov function that is used to establish geometric ergodicity in Theorem \ref{thm:ergodicity}. For the reader's convenience, we recall the definition of a Lyapunov function below.

\begin{definition}[Lyapunov Function] \label{def:lyapunov}A function $V(x,y)\in C^2(\R;\rbb)$ is a Lyapunov function on $\R\subset \mathbb{R}^2$ if the following hold:

1. $V(x,y)\ge 0, (x,y)\in \mathcal{R}$ and $V(x,y)\to \infty$  whenever $| (x,y)|\to \infty$ in $\mathcal{R}$; and

2. there exist positive constants $a_1$ and $a_2$ such that
\begin{equation}\label{cond:Lv<-v}
\mathcal{L}V(x,y)<-a_1V(x,y)+a_2,\quad (x,y)\in \mathcal{R},
\end{equation}
where $\L$ is the generator given by \eqref{form:L}. 

In case, $\R=\rbb^2$, $V$ is called a global Lyapunov function. Otherwise, $V$ is called a local Lyapunov function.
\end{definition}

Our main result in this section is the following lemma giving the existence of a globally Lyapunov function.
\begin{lemma} \label{lem:Lyapunov} There exists a global Lyapunov function for system \eqref{eq:system}.
\end{lemma}

Following the approach in \cite{athreya2012propagating,birrell2012transition,
cooke2017geometric,foldes2021sensitivity,
herzog2015noiseI,
herzog2015noiseII,
herzog2017ergodicity,
kolba2019noise}, the proof of Lemma \ref{lem:Lyapunov} consists of two main steps: we first construct local Lyapunov functions on different regions of the phase space. This will be heuristically explained using a scaling analysis in Section \ref{sec:lyapunov:intuition} whereas the rigorous proofs will be presented in Section \ref{sec:lyapunov:R1R2R3}. Then, we will patch them altogether allowing us to obtain a single global Lyapunov function. The gluing argument will be carried out in Section \ref{sec:lyapunov:gluing}.

\subsection{Heuristics and decomposition of the phase space} \label{sec:lyapunov:intuition}

Before diving into the precise details of the Lyapunov construction, we build some heuristic about the system \eqref{eq:system}. This will help gain a better understanding of the Lyapunov construction employed in \cite{kolba2019noise}.

Firstly, to see intuitively the dynamics in different regions of $\rbb^2$, we provide some numerics in Figure \ref{fig:simulations}.

\begin{figure}[htb]
\renewcommand{\arraystretch}{0.1}
\setlength{\tabcolsep}{-4pt}
\begin{tabular}{ccc}
  \includegraphics[width=0.3\linewidth]{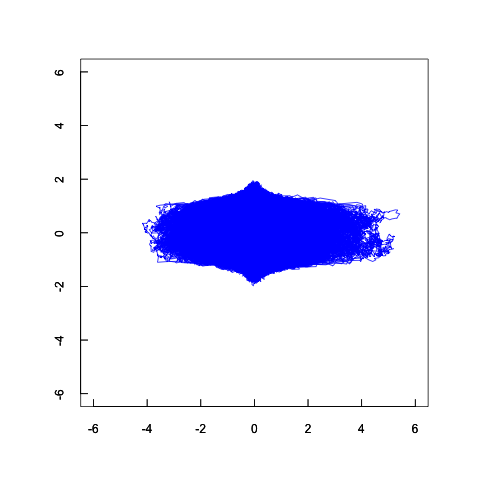}\vspace*{-1mm} &   \includegraphics[width=0.3\linewidth]{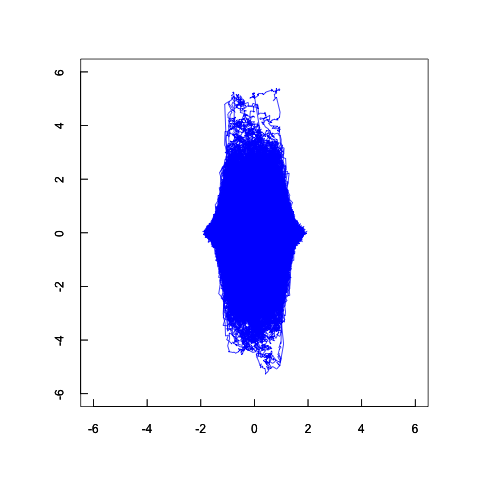}\vspace*{-1mm} & 
  \includegraphics[width=0.3\linewidth]{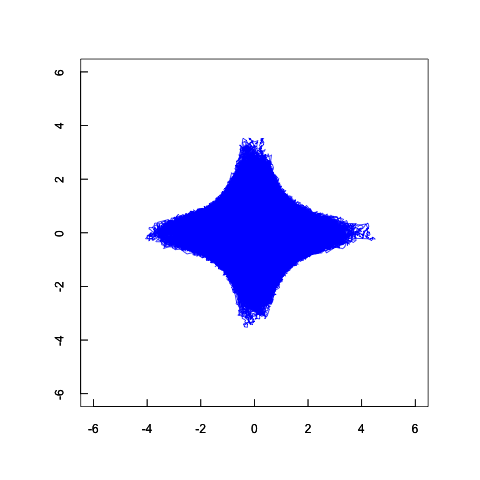}\vspace*{-1mm} \\
$m=2,n=9$, $h'>0$ & $m=9,n=2$, $h'>0$  & $m=n=5$, $h'>0$ \\[1pt]
  \includegraphics[width=0.3\linewidth]{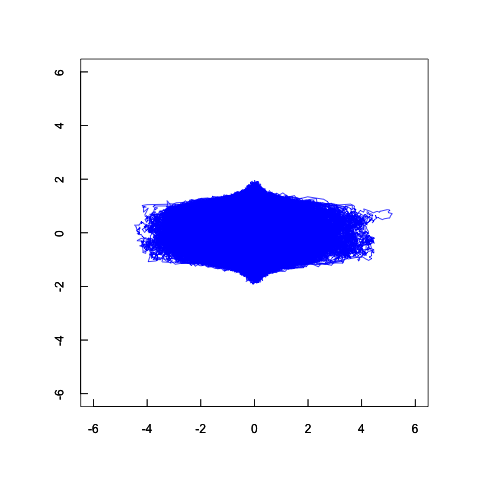}\vspace*{-1mm} &   \includegraphics[width=0.3\linewidth]{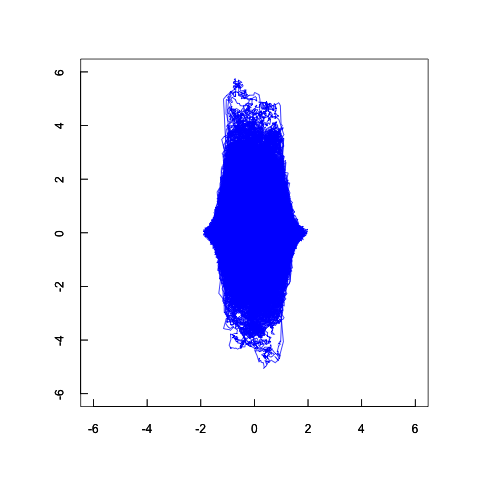}\vspace*{-1mm} & 
  \includegraphics[width=0.3\linewidth]{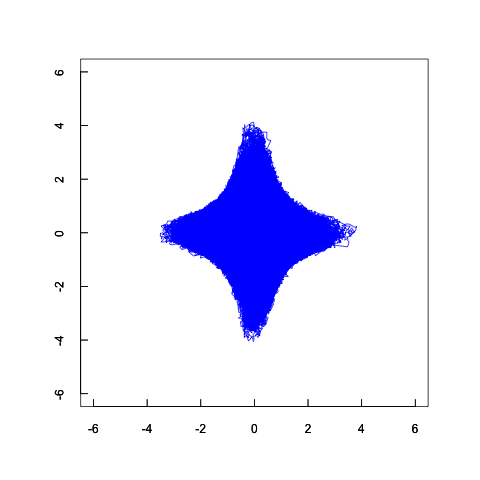}\vspace*{-1mm} \\
$m=2,n=9$, $h'<0$ & $m=9,n=2$, $h'<0$  & $m=n=5$, $h'<0$ \\[1pt]
\end{tabular}
\caption{Simulation of the system with $\epsilon_x=\epsilon_y=10,q=2,m=n,h(t)=\pm t$.}
\label{fig:simulations}
\end{figure}
The numerics indicate that there are three distinct regions where the dynamics behave differently, namely, when $|x|$ and $|y|$ are both large, when $|x|\to\infty$ while $|y|\to 0$, and when $|x|\to0$ while $|y|\to \infty$. 

To further determine the boundaries as well as local Lyapunov functions on each subregions, we will assume for the sake of simplicity that
\begin{align*}
h(t)=t,
\end{align*}
So that
\begin{align*}
h'(t)=1.
\end{align*}
Concerning the first region when $|x|$ and $|y|$ are both large, we introduce the following scaling transformation
\begin{align*}
S_1:(x,y)\to (\lambda x,\lambda y),
\end{align*}
for a parameter $\lambda >0$. Considering $\L$ as in \eqref{form:L}, under the transformation $S_1$, we obtain (here $h'=1$)
\begin{align*}
\L \circ S_1(x,y)
& = \lambda^{m+n-2}nx^{m}y^{n-1}\partial_x- \lambda^{q(m+n-2)}nx^{q(m-1)+1}y^{q(n-1)}\partial_x \\
&\qquad-\lambda^{m+n-2}mx^{m-1}y^{n}\partial_y- \lambda^{q(m+n-2)}mx^{q(m-1)}y^{q(n-1)+1}\partial_y\\
&\qquad+\frac{1}{2}\epsilon_x^2\lambda^{-2}\partial_{xx} +\frac{1}{2}\epsilon_y^2\lambda^{-2}\partial_{yy}.
\end{align*}
We observe that as $\lambda\to\infty$ (while $x,y$ are being fixed), the first order terms in the above right-hand side are the dominating quantities, thanks to the exponent $\lambda^{m+n-2}$. Together with the numerics in Figure \ref{fig:simulations}, this suggests that the region for large $|x|$ and $|y|$ is given by
\begin{align*}
\R_1=\{(x,y)\in\rbb^2: |x|^{m-1}|y|^{n-1} >c\},
\end{align*}
for some suitably chosen constant $c$. In this region, a natural candidate for the Lyapunov function is given by the norm of $(x,y)$, i.e., $$V_1(x,y)=x^2+y^2.$$ 
Later in Lemma \ref{lem:v1}, we will show that such $V_1$ indeed satisfies the conditions of Definition \ref{def:lyapunov}.

With regard to the second subregion where $|x|\to\infty$ while $|y|\to 0$, namely,
\begin{align*}
\R_2=\{(x,y)\in\rbb^2: |x|^{m-1}|y|^{n-1} <c, |x|>C\},
\end{align*}
let $S_2$ be the transformation
\begin{align*}
S_2:(x,y)\to (\lambda^{\frac{1}{m-1}} x,\lambda^{-\frac{1}{n-1}} y).
\end{align*}
Under this transformation, we have
\begin{align*}
\L \circ S_2(x,y)
& = nx^{m}y^{n-1}\partial_x- nx^{q(m-1)+1}y^{q(n-1)}\partial_x -mx^{m-1}y^{n}\partial_y- mx^{q(m-1)}y^{q(n-1)+1}\partial_y\\
&\qquad+\frac{1}{2}\epsilon_x^2\lambda^{-\frac{2}{m-1}}\partial_{xx} +\frac{1}{2}\epsilon_y^2\lambda^{\frac{2}{n-1}}\partial_{yy}.
\end{align*}
Observe that in this situation, as $\lambda\to \infty$, the above right-hand side is dominated by the second order term $\partial_{yy}$. This suggests that a Lyapunov function $V_2(x,y)$ in $\R_2$ satisfies
\begin{align*}
\begin{cases}\partial_{yy}V_2(x,y)\propto -x^2,\\
\,\,\,\,\,V_2|_{\R_1\cap \R_2}\approx x^2.\end{cases}
\end{align*}
We note that the last equation above stems from the fact that $$V_1(x,y)\approx x^2,\quad (x,y)\in \R_1\cap\R_2.$$
A candidate of the above system has the form
\begin{align*}
V_2(x,y)=cx^2(1- \tilde{c} y^2),
\end{align*}
for some suitably chosen constants $c$ and $\tilde{c}$. Later in Lemma \ref{lem:v2}, we will provide the explicit choice of $V_2$ and prove that $V_2$ indeed satisfies the condition of Definition \ref{def:lyapunov} in $\R_2$.

Turning to the last subregion  where $|x|\to0$ while $|y|\to \infty$, namely,
\begin{align*}
\R_3=\{(x,y)\in\rbb^2: |x|^{m-1}|y|^{n-1} <c, |y|>C\}.
\end{align*}
Similarly to region $\R_2$, we consider the transformation
\begin{align*}
S_3:(x,y)\to (\lambda^{-\frac{1}{m-1}} x,\lambda^{\frac{1}{n-1}} y).
\end{align*}
A routine computation gives
\begin{align*}
\L \circ S_3(x,y)
& = nx^{m}y^{n-1}\partial_x- nx^{q(m-1)+1}y^{q(n-1)}\partial_x 
 -mx^{m-1}y^{n}\partial_y- mx^{q(m-1)}y^{q(n-1)+1}\partial_y\\
&\qquad+\frac{1}{2}\epsilon_x^2\lambda^{\frac{2}{m-1}}\partial_{xx} +\frac{1}{2}\epsilon_y^2\lambda^{-\frac{2}{n-1}}\partial_{yy}.
\end{align*}
In turn, the dominating balance of force in this region is contained in the second order term $\partial_{xx}$. This implies that the Lyapunov function should satisfy
\begin{align*}
\begin{cases}\partial_{xx}V_3(x,y)\propto -y^2,\\
\,\,\,\,\,V_3|_{\R_1\cap \R_3}\approx y^2,\end{cases}
\end{align*}
where the last condition above follows from the fact that
$$V_1(x,y)\approx y^2,\quad (x,y)\in \R_1\cap\R_3.$$
As an analogue of $V_2$, the candidate $V_3$ in $\R_3$ is given by
\begin{align*}
V_3(x,y)=cy^2(1- \tilde{c} x^2),
\end{align*}
for some suitably chosen constants $c$ and $\tilde{c}$. The explicit choice of these parameters will be provided in Lemma \ref{lem:v3} where we verify the condition of Definition \ref{def:lyapunov}.

\begin{remark} \label{rem:kolba2} As mentioned in the introduction, we note that the expressions of $V_i$, $i=1,2,3,$ are essentially the same as those discussed in \cite{kolba2019noise}. However, we will have to choose the parameters more carefully so as to achieve the dissipative bound \eqref{cond:Lv<-v}, which was not available in \cite{kolba2019noise}. 
\end{remark}

Having obtained the Lyapunov function on each subregions, in Section \ref{sec:lyapunov:gluing}, we will ``patch" the $V_i$, $i=1,2,3$ on overlapping regions to create a global Lyapunov function. We finish this section by the following definition of $\R_1,\R_2$ and $\R_3$.
\begin{definition} \label{def:R_123} The regions $\R_1,\R_2$ and $\R_3$ are given by
\begin{align}
     \mathcal{R}_1&:=\{(x,y)\in \mathbb{R}^2: |x|^{m-1}|y|^{n-1}\ge c_1\}\label{eq:R1}\\
      \mathcal{R}_2&:=\{(x,y)\in \mathbb{R}^2: |x|^{m-1}|y|^{n-1}\le 2c_1, |x|\ge c_2\}\label{eq:R2}\end{align}
      and
\begin{align}
\mathcal{R}_3&:=\{(x,y)\in \mathbb{R}^2: |x|^{m-1}|y|^{n-1}\le 2c_1, |y|\ge c_3\}\label{eq:R3},
\end{align}
for some positive constants $c_1,c_2$ and $c_3$ to be chosen later.
\end{definition}

\begin{figure}[htb]
  \includegraphics[width=0.6\linewidth,
  trim={0.3cm 0.3cm 0.3cm 0.3cm},clip]{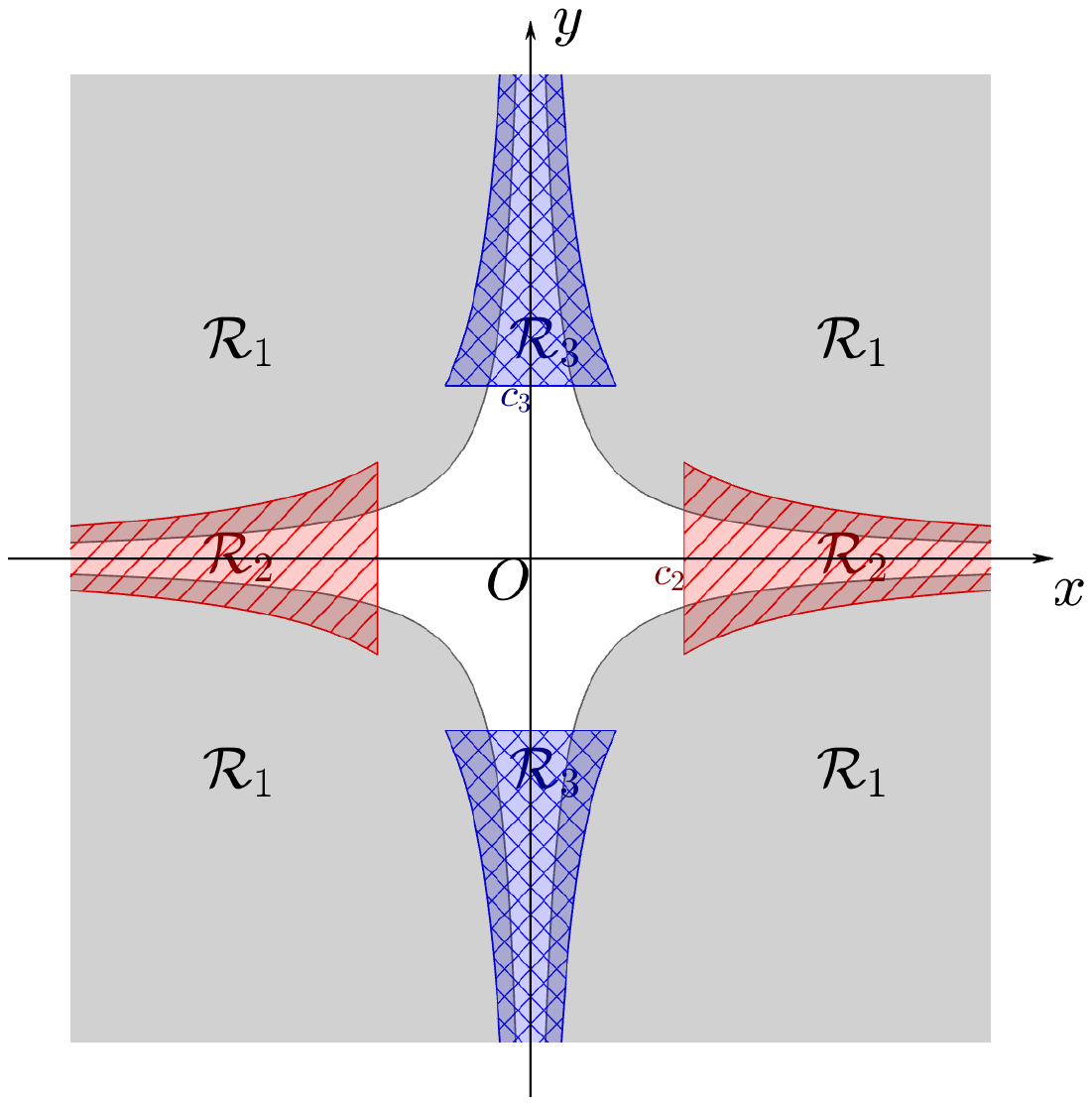}
  \vspace*{-3mm}
  \caption{Decomposition of the plane into $\mathcal{R}_1,\mathcal{R}_2,\mathcal{R}_3$.}
  \label{fig:decomposition}
\end{figure}

\begin{remark}
1. We note that in \cite{kolba2019noise}, the authors choose $c_2=c_3=1$ and $c_1>\frac{2}{a}$ where $a$ is the constant as in Assumption \ref{cond:h:h'>a}. Although these choices are good enough to establish \eqref{ineq:LV->-infty}, they are not sufficient to produce the correct dissipative bound as in \eqref{cond:Lv<-v}. In our work, we are able to circumvent this issue by taking $c_1,c_2$ and $c_3$ sufficiently large.

2. We also would like to mention that we carry out the above heuristic using $h(t)=t$ in order to look for a simple decomposition of $\rbb^2$. In turn, the resulting PDEs are relatively not difficult to solve, while still allowing us to achieve a global control. Although the presented calculation may give the impression that the regions $\mathcal{R}_i$, $i=1,2,3$ from Definition \ref{def:R_123} and the choice of $V_i$, $i=1,2,3$ are only valid for $h(t)=t$, they can actually be leveraged to construct Lyapunov functions for any $h$ satisfying Assumption \ref{cond:h} and parameters $m,n,q$ satisfying Assumption \ref{cond:m_and_n}. This will be rigorously justified in Sections \ref{sec:lyapunov:R1R2R3} and \ref{sec:lyapunov:gluing} where we supply the Lyapunov proofs in detail. Having said that, we note that $\mathcal{R}_i$ and $V_i$, $i=1,2,3$, need not be the unique choice for \eqref{eq:system}.
\end{remark}

\subsection{Local Lyapunov functions in $\mathcal{R}_1,\mathcal{R}_2$ and $\mathcal{R}_3$} \label{sec:lyapunov:R1R2R3}

For notational convenience, we denote
\begin{equation}\label{eq:u}
    u(x,y):=|w(x,y)|=|h'(x^my^n)||x|^{m-1}|y|^{n-1},
\end{equation}
where $w$ is as in \eqref{eq:w} and $h$ is as in Assumption \ref{cond:h}. In Lemma \ref{lem:v1}, stated and proven next, we provide the Lyapunov bound in region $\R_1$. The proof of Lemma \ref{lem:v1} is a slight reworking of that of \cite[Lemma 2]{kolba2019noise}.
\begin{lemma}\label{lem:v1}
         Let $V_1$ be defined as
         \begin{equation}\label{eq:v1}
             V_1(x,y)=x^2+y^2.
         \end{equation} 
         Then, for all \begin{equation}\label{eq:c1}
        c_1>\frac{4^{\frac{1}{q-1}}}{a},
         \end{equation}
where $a$ is as in \eqref{cond:h:h'>a}, the following holds: \begin{equation}\label{ub:lv1}
             \mathcal{L}V_1\le -\frac{1}{2}V_1-u^qV_1+\epsilon_x^2+\epsilon_y^2,\quad (x,y)\in\R_1,
         \end{equation}
         where $u$ and $\R_1$ are given by \eqref{eq:u} and \eqref{eq:R1}, respectively. In particular, $V_1$ is a  Lyapunov function on $\mathcal{R}_1$.
\end{lemma}
\begin{proof}
Recalling $\mathcal{L}$, $w$ and  $u=u(x,y)=|w(x,y)|$ as in \eqref{form:L}, \eqref{eq:w} and \eqref{eq:u}, respectively, a routine computation gives
\begin{align}\label{eq:lv1}
    (\mathcal{L}V_1)(x,y)&=2nx^2(w-|w|^q)+\epsilon_x^2+2my^2(-w-|w|^q)+\epsilon_y^2 \notag \\
    &=  2nx^2(w-u^q)+2my^2(-w-u^q)+\epsilon_x^2+\epsilon_y^2.
\end{align} 
Letting $c_1>4^{\frac{1}{q-1}}/a$ where $a$ is as in \eqref{cond:h:h'>a}, we observe that for $(x,y)\in\R_1$,
\begin{align} \label{ub:u>4^(1/(q-1))_in_R1}
u=\left|h'\left(x^my^n\right)\right||x|^{m-1}|y|^{n-1}>a\cdot c_1>a\cdot \frac{4^{\frac{1}{q-1}}}{a}=4^{\frac{1}{q-1}},\quad (x,y)\in\R_1.
\end{align}  
In the second estimate above, we employed the fact that $|h'|>a$ by virtue of \eqref{cond:h:h'>a}. As a consequence, the following holds in $\R_1$
\begin{equation}\label{eq:bound on u}
   \quad u-\frac{1}{2}u^q<-\frac{1}{4}.
\end{equation}
Turning back to \eqref{eq:lv1}, since $-u\le w\le u$ and $u-u^q<0$, we have
\begin{align*}
        \mathcal{L}V_1&\le 2(nx^2+my^2)(u-u^q)+\epsilon_x^2+\epsilon_y^2\\
            &\le 2(x^2+y^2)\left(u-u^q\right)+\epsilon_x^2+\epsilon_y^2,
            \end{align*}
            where the last estimate follows from the condition that $m>1$ and $n>1$, cf. Assumption \ref{cond:m_and_n}. Together with the estimate of \eqref{eq:bound on u}, we further deduce
            \begin{align*}
                   \mathcal{L}V_1     &\le 2\left(u-\frac{1}{2}u^q\right)(x^2+y^2)-u^q(x^2+y^2)+\epsilon_x^2+\epsilon_y^2  \\
            &\le- \frac{1}{2}V_1-u^qV_1+\epsilon_x^2+\epsilon_y^2. 
\end{align*}
This produces \eqref{ub:lv1}, establishing that
$V_1$ is a local Lyapunov function in $\mathcal{R}_1$, by Definition \ref{def:lyapunov}. 
\end{proof}

Next, we consider region $\R_2$ given by \eqref{eq:R2}. In view of the heuristic argument in Section \ref{sec:lyapunov:intuition}, we have the following result providing a Lyapunov function in $\R_2$. 

\begin{lemma}\label{lem:v2}
    Fixing
    \begin{equation}\label{eq:k2}
        k_2=\big(4^{\frac{q}{q-1}}+1\big)\frac{n}{\epsilon_y^2},
    \end{equation}
   let $V_2$ be defined as
    \begin{equation}\label{eq:v2}
        V_2(x,y)=x^2(1-k_2y^2).
    \end{equation}
Suppose the positive constant $c_2$ as in \eqref{eq:R2} satisfies the following choice
    \begin{equation}\label{ub:c2}
        c_2= \Big(2c_1\cdot\Big(\frac{n}{bk_2(m+n)}\Big)^\frac{1-m}{2}\Big)^\frac{1}{n-1},
    \end{equation}
    where $b>4$, $c_1$ is as in \eqref{eq:R1}. Then, it holds that
    \begin{equation}\label{ub:lv2}
         (\mathcal{L}V_2)(x,y)\le-V_2-\frac{1}{2}x^2u^q+\epsilon_x^2,\quad (x,y)\in \mathcal{R}_2,
    \end{equation}
    where  $\mathcal{R}_2$ is given by \eqref{eq:R2}.
    In particular, $V_2$ is a Lyapunov function on $\mathcal{R}_2$ as in Definition~\ref{def:lyapunov}.
\end{lemma}
\begin{proof}
 To see that $V_2$ satisfies the first condition of Definition \ref{def:lyapunov}, we recall from \eqref{eq:R1} and \eqref{eq:R2} that for every $(x,y)\in \mathcal{R}_2$, $|x|^{m-1}|y|^{n-1}\le 2c_1$ and $|x|\ge c_2$. Denoting
\begin{equation}\label{eq:C2}
   C_2:=\left(\frac{2c_1}{c_2^{m-1}}\right)^{\frac{2}{n-1}},
\end{equation}  
observe that
\begin{align*}
y^2\le C_2,\quad (x,y)\in\mathcal{R}_2.
\end{align*}
In particular, given the choice of $c_2$ defined in \eqref{ub:c2}, we have
\begin{equation}
\label{lem3.2:key obs 3}
   y^2\le C_2= \frac{n}{bk_2(m+n)},
\end{equation}
whence (for $b>4$)
\begin{align*}
1\ge 1-k_2y^2\ge 1-\frac{n}{b(m+n)}\ge \frac{3}{4}.
\end{align*}
It follows that in $\R_2$, $V_2>3x^2/4$ implying $V_2\to \infty$ when $|x|\to\infty$ in $\R_2$.

Turning to the second condition of Definition \ref{def:lyapunov}, recall $\mathcal{L}$ and $w$ as in \eqref{form:L} and \eqref{eq:w}, respectively. Applying $\L$ to $V_2$ gives
\begin{equation}\label{eq:lv2}
    (\mathcal{L}V_2)(x,y)=2nx^2\left(1-k_2y^2\right)\left(w-w^q\right)+2mk_2x^2y^2\left(w+w^q\right)+\epsilon_x^2\left(1-k_2y^2\right)-\epsilon_y^2k_2x^2.
\end{equation}
Letting $u\ge 0$ be defined in \eqref{eq:u}, i.e., $u=|w|$ for all $(x,y)\in \mathbb{R}^2$, it follows from \eqref{eq:lv2} that
\begin{align} \label{ineq:Kolba-remark}
(\mathcal{L}V_2)(x,y)
      &\le 2nx^2\left(1-k_2y^2\right)\left(u-u^q\right)+2mk_2x^2y^2\left(u+u^q\right)+\epsilon_x^2-\epsilon_y^2k_2x^2 .
\end{align}
whence
\begin{align}
 (\mathcal{L}V_2)(x,y)
      &\le 2x^2\left[\left(k_2(m+n)y^2-n\right)u^q+\left(k_2(m+n)y^2+n\right)u\right]-\epsilon_y^2k_2x^2+\epsilon_x^2  \nt \\
    &=2x^2A_2-\frac{n}{2}x^2u^q-\epsilon_y^2k_2x^2+\epsilon_x^2,\quad (x,y)\in \mathcal{R}_2, \label{ub:lv2.1}
\end{align}
where
\begin{equation*} 
    A_2:=\left(k_2(m+n)y^2-\frac{3}{4}n\right)u^q+\left(k_2(m+n)y^2+n\right)u.
\end{equation*}
Since $y^2\le C_2$ in $\R_2$, observe that
\begin{equation}
    A_2\le \left(k_2(m+n)C_2-\frac{3}{4}n\right)u^q+\left(k_2(m+n)C_2+n\right)u.
\end{equation}
From \eqref{ub:c2} and \eqref{eq:C2}, since $ k_2(m+n)C_2=\frac{n}{b}$,
we recast the above inequality as
\begin{align}\label{ineq:A_2<g}
A_2\le n\cdot g(u),
\end{align}
where
\begin{align} \label{form:g(u)}
g(u)= \Big(\frac{1}{b}-\frac{3}{4}\Big)u^q+\Big(\frac{1}{b}+1\Big)u.
\end{align}
Furthermore, for all $b>4$, observe that $g(u)$ is concave down on $u\in[0,\infty)$ since $q>1$. As a consequence, a straight-forward calculation gives
\begin{align*}
\underset{u\ge 0}{\text{argmax }} g(u) = \bigg[ \frac{4(\frac{1}{b}+1)}{q(3-\frac{4}{b})} \bigg]^{\frac{1}{q-1}},
\end{align*}
and
\begin{align*}
\max_{u\ge 0}g(u)=\left[\Big(\frac{1}{b}+1\Big) \bigg[ \frac{4(\frac{1}{b}+1)}{q(3-\frac{4}{b})} \bigg]^{\frac{1}{q-1}} -\Big(\frac{3}{4}-\frac{1}{b}\Big)\bigg[ \frac{4(\frac{1}{b}+1)}{q(3-\frac{4}{b})} \bigg]^{\frac{q}{q-1}}  \right].
\end{align*}
In particular, choosing $b>4$, we further deduce
\begin{align} \label{ineq:max_g(u)}
\max_{u\ge 0}g(u)=\Big(\frac{1}{b}+1\Big) \bigg[ \frac{4(\frac{1}{b}+1)}{q(3-\frac{4}{b})} \bigg]^{\frac{1}{q-1}}\le  2\cdot 4^{\frac{1}{q-1}}.
\end{align}
This together with \eqref{ineq:A_2<g} implies the bound
\begin{align} \label{lem3.2:key obs 1}
A_2\le  n\cdot 2\cdot 4^{\frac{1}{q-1}}.
\end{align}
Turning back to \eqref{ub:lv2.1}, we estimate $\L V_2$ making use of \eqref{lem3.2:key obs 1} and the choice of $k_2$ as in \eqref{eq:k2} as follows:
\begin{align*}
\L V_2\le 2A_2x^2 -\frac{n}{2}x^2u^2-\epsilon_y^2k_2x^2+\epsilon_x^2 
   &\le 2\cdot n\cdot 2\cdot 4^{\frac{1}{q-1}}\cdot x^2-\frac{n}{2}x^2u^2-\epsilon_y^2\Big(4^{\frac{q}{q-1}}+1\Big) \frac{n}{\epsilon_y^2}x^2+\epsilon_x^2 \nt \\
   &=-nx^2-\frac{n}{2}x^2u^2+\epsilon_x^2 \nt \\
   &\le -nx^2(1-k_2y^2)-\frac{n}{2}x^2u^2+\epsilon_x^2. 
\end{align*}
Employing the fact that $n>1$, this produces the dissipative bound \eqref{ub:lv2}. Therefore, $V_2$ is a Lyapunov function on $\mathcal{R}_2$, as claimed.

\end{proof}

\begin{remark} \label{rem:Kolba} We note that the proof of Lemma \ref{lem:v2} is an improvement of that of \cite[Lemma 3]{kolba2019noise} giving a Lyapunov function for region $\R_2$. In the proof of \cite[Lemma 3]{kolba2019noise}, when $q=2$, the term $x^2y^2\left(u+u^q\right)=x^2y^2\left(u+u^2\right)$ on the right-hand side of \eqref{ineq:Kolba-remark} is considered to be negligible in $\R_2$. This may be possible if we assume $m=n$ (recalling $u$ defined in \eqref{eq:u}). In general, since we do not impose a growth rate on $h'$, for arbitrarily $m> n$, it is not immediately clear that we can omit this term. We therefore take it into account as presented in the proof of Lemma \ref{lem:v2}.
\end{remark}

Next, in Lemma \ref{lem:v3}, we establish a Lyapunov bound in region $\R_3$. The proof of Lemma \ref{lem:v3} is similar to that of Lemma \ref{lem:v2}.

\begin{lemma}\label{lem:v3}
     Fixing \begin{equation}\label{eq:k3}k_3=\Big(4^{\frac{q}{q-1}}+1\Big) \frac{m}{\epsilon_x^2},
    \end{equation} let $V_3$ be defined as
    \begin{equation}\label{eq:v3}
        V_3(x,y)=y^2(1-k_3x^2).
    \end{equation}
 Suppose the positive constant $c_3$ as in \eqref{eq:R3} satisfies the following choice   
    \begin{equation}\label{ub:c3}
        c_3=\Big(2c_1\cdot\Big(\frac{m}{bk_3(m+n)}\Big)^\frac{1-m}{2}\Big)^\frac{1}{n-1},
    \end{equation} 
    where $b>4$ and $c_1>0$ is as in \eqref{eq:R1}. Then, the following holds
    \begin{equation}\label{ub:lv3}
         (\mathcal{L}V_3)(x,y)<-V_3-\frac{1}{2}y^2u^2+\epsilon_y^2,\quad (x,y)\in \mathcal{R}_3,
    \end{equation}
  where  $\mathcal{R}_3$ is given by \eqref{eq:R3}.    
     In particular, $V_3$ is a Lyapunov function on $\mathcal{R}_3$ as in Definition \ref{def:lyapunov}.
\end{lemma}
\begin{proof} Firstly, we proceed to verify that the condition 1 of Definition \ref{def:lyapunov} for $V_3$. Recall $\R_3$ from \eqref{eq:R3} that for all $(x,y)\in \mathcal{R}_3$, $|x|^{m-1}|y|^{n-1}\le 2c_1$ and $|y|\ge c_3$ where $c_1$ is the boundary threshold in \eqref{eq:R1}. Denoting
\begin{equation}\label{eq:C3}
C_3:=\left(\frac{2c_1}{c_3^{n-1}}\right)^{\frac{2}{m-1}},
\end{equation}
observe that
\begin{align*}
x^2\le \Big(\frac{2c_1}{|y|^{n_1}}\Big)^{\frac{1}{m-1}} \le C_3,\quad (x,y)\in\mathcal{R}_3.
\end{align*}
Letting $c_3$ satisfy \eqref{ub:c3}, we obtain
\begin{align} \label{ineq:x^2<C_3}
x^2\le C_3=\frac{m}{bk_3(m+n)}.
\end{align}
Picking $b>4$, we further deduce
\begin{align*}
1-k_3x^2\ge 1-\frac{m}{4(m+n)}\ge \frac{3}{4}.
\end{align*}
It follows that
\begin{align} \label{ineq:V_2>y^2:R_3}
V_3=y^2(1-k_3x^2) \ge \frac{3}{4}y^2.
\end{align}
Hence, $V_3$ tends to infinity whenever $|y|\to \infty$ in $\R_3$. This verifies the first condition of Definition \ref{def:lyapunov}.

Turning to the dissipative bound \eqref{ub:lv3}, we apply $\L$ as in \eqref{form:L} to $V_3$ and obtain
:\begin{equation}\label{eq:lv3}(\mathcal{L}V_3)(x,y)=2my^2\left(1-k_3x^2\right)\left(-w-w^q\right)+2nk_3x^2y^2\left(w^q-w\right)+\epsilon_y^2\left(1-k_3x^2\right)-\epsilon_x^2k_3y^2,
\end{equation}
where $w$ is defined in \eqref{eq:w}. Similarly to the proof of Lemma \ref{lem:v2}, letting $u=|w|$, we estimate the right-hand side of \eqref{eq:lv3} as follows.
\begin{align}
(\mathcal{L}V_3)(x,y)
      &\le 2my^2\left(1-k_3x^2\right)\left(u-u^q\right)+2nk_3x^2y^2\left(u^q+u\right)+\epsilon_y^2-\epsilon_x^2k_3y^2 \nt \\
      &= 2my^2u-2my^2u^q+2mk_3x^2y^2u+2mk_3x^2y^2u^q\nt \\
      &\quad+2nk_3x^2y^2u^q+2nk_3x^2y^2u+\epsilon_y^2-\epsilon_x^2k_3y^2 \nt \\
    &=2A_3y^2-\frac{m}{2}y^2u^q-\epsilon_x^2k_3y^2+\epsilon_y^2,\quad (x,y)\in \mathcal{R}_3,\label{ub:lv3.1}
\end{align}
where 
\begin{equation*}
    A_3:=\left(k_3(m+n)x^2-\frac{3}{4}m\right)u^q+\left(k_3(m+n)x^2+m\right)u.
\end{equation*}
From \eqref{ineq:x^2<C_3}, we note that
\begin{align}\label{eq:g3}
    A_3&\le \left(k_3(m+n)C_3-\frac{3}{4}m\right)u^2+\left(k_3(m+n)C_3+m\right)u \nt \\
    &=m\left(\frac{1}{b}-\frac{3}{4}\right)u^2+m\left(\frac{1}{b}+1\right)u=m\cdot g(u),
\end{align}
where $g(u)$ is as in \eqref{form:g(u)}. In view of \eqref{ineq:max_g(u)}, we obtain
\begin{align*}
A_3<m\cdot 2\cdot 4^{\frac{1}{q-1}}.
\end{align*}
Together with the bound \eqref{ub:lv3.1} and the choice of $k_3$ as in \eqref{eq:k3}, we infer for all $(x,y)\in\R_3$,
\begin{align*}
\L V_3\le 2A_3y^2 -\frac{m}{2}y^2u^2-\epsilon_x^2k_3y^2+\epsilon_y^2  
   &\le 2m\cdot 2\cdot 4^{\frac{1}{q-1}} y^2-\frac{m}{2}y^2u^2-\epsilon_x^2\Big(4^{\frac{q}{q-1}}+1\Big) \frac{m}{\epsilon_x^2}y^2+\epsilon_y^2\\
   &=-my^2-\frac{m}{2}y^2u^2+\epsilon_y^2.
\end{align*}
Recalling $m>1$ by virtue of Assumption \ref{cond:m_and_n}, this produces the bound \eqref{ub:lv3}, thereby verifying the second condition of Definition \ref{def:lyapunov}. The proof is thus finished.

\end{proof}

\subsection{Local Lyapunov functions in overlapping regions} \label{sec:lyapunov:gluing} Given the local Lyapunov functions $V_i$, $i=1,2,3$, in this section, we proceed to glue them in the over-lapping regions $\R_1\cap\R_2$ and $\R_1\cap\R_3$ to create a single globally Lyapunov function. For this purpose, we introduce the following smooth cut-off function
$\phi:\mathbb{R}\to \mathbb{R}$ given by
\begin{equation}\label{form:phi}
    \phi(t)=\begin{cases}
        1,&|t|\ge 4,\\
        \text{monotone},& 1<t<4,\\
        0,& |t|\le 1.
    \end{cases}
\end{equation}
Denoting
\begin{equation}\label{eq:lambda}
    \lambda(x,y):=\left(\frac{|x|^{m-1}|y|^{n-1}}{c_1}\right)^2.
\end{equation}
for $i=2,3$, we define $V_{1i}:\mathcal{R}_1\cap \mathcal{R}_i\to \mathbb{R}$ as:
\begin{equation}\label{eq:v1i}
    V_{1i}(x,y)=
    \phi\big(\lambda(x,y)\big)\cdot V_1(x,y)
    +
    \left[1-\phi\big(\lambda(x,y)\big)\right]\cdot V_i(x,y).
\end{equation}
For the sake of convenience, in what follows, we compute the partial derivative terms on the right-hand side of \eqref{eq:Lv13}. We will make use of these identities to establish the Lyapunov property of $V_{1i}$ in $\R_1\cap\R_i$, $i=2,3$. The first derivatives are given by
\begin{align}
\frac{\partial V_{1i}} {\partial x}
&=\phi'\cdot\lambda \cdot\frac{2(m-1)}{x}(V_1-V_i)+\phi\cdot\frac{\partial V_1}{\partial x}+(1-\phi)\cdot\frac{\partial V_i}{\partial x},\label{eq:dx}
\end{align}
and
\begin{align}
\frac{\partial V_{1i}}{\partial y} 
&=\phi'\cdot\lambda \cdot\frac{2(n-1)}{y}(V_1-V_i)+\phi\cdot\frac{\partial V_1}{\partial y}+(1-\phi)\cdot\frac{\partial V_i}{\partial y}.\label{eq:dy}
\end{align}
The expressions of the second derivatives are provided next
\begin{align}
\frac{\partial^2 V_{1i}}{\partial x^2}&=
\phi^{''}\cdot\lambda^2\cdot\frac{2^2(m-1)^2}{x^2}(V_1-V_i)+\phi'\cdot\lambda \cdot\frac{2^2(m-1)^2}{x^2}(V_1-V_i)\notag\\
&\qquad+\phi'\cdot \lambda \cdot \frac{2(1-m)}{x^2}(V_1-V_i)
    +
    2\phi'\cdot \lambda\cdot \frac{2(m-1)}{x}\frac{\partial (V_1-V_i)}{\partial x}\notag\\
    &\qquad+\phi\cdot\frac{\partial^2 V_1}{\partial x^2}
    +
    (1-\phi)\cdot\frac{\partial^2 V_i}{\partial x^2},
\label{eq:dx2}
\end{align}
and
\begin{align}
 \frac{\partial^2 V_{1i}}{\partial y^2}
& =\phi^{''}\cdot\lambda^2\cdot \frac{2^2(n-1)^2}{y^2}\cdot (V_1-V_i)
    +
    \phi'\cdot \lambda \cdot\frac{2^2(n-1)^2}{y^2} (V_1-V_i) 
    \notag\\
    &\qquad+\phi'\cdot \lambda \cdot \frac{2(1-n)}{y^2}(V_1-V_i) 
    +2\phi'\cdot \lambda\cdot \frac{2(n-1)}{y}\frac{\partial (V_1-V_i)}{\partial y}
    \notag\\
    &\qquad+\phi\cdot\frac{\partial^2 V_1}{\partial y^2}
    +
    (1-\phi)\cdot \frac{\partial^2 V_i}{\partial y^2}.
\label{eq:dy2}
\end{align}

We now proceed to verify that $V_{1i}$, $i=2,3$, is a Lyapunov function in the overlapping region $\R_1\cap \R_i$. Ultimately, the results below paired with Lemma \ref{lem:v1}, Lemma \ref{lem:v2} and Lemma \ref{lem:v3} create a single globally Lyapunov function for system \eqref{eq:system}.

\begin{lemma}\label{lem:v12}Let $V_{12}$ be defined in \eqref{eq:v1i} (with $i=2$), $c_1,c_2$ and $k_2$ be the constants in Lemma \ref{lem:v2}. Then, for all $c_1$ sufficiently large, $V_{12}$ is a local Lyapunov function in $\mathcal{R}_1\cap\mathcal{R}_2$.
\end{lemma}

\begin{lemma}\label{lem:v13}Let $V_{13}$ be defined in \eqref{eq:v1i} (with $i=3$), $c_1,c_3$ and $k_3$ be the constants in Lemma \ref{lem:v3}. Then, for all $c_1$ sufficiently large, $V_{13}$ is a local Lyapunov function in $\mathcal{R}_1\cap\mathcal{R}_3$.
\end{lemma}

To avoid repetition, we only present the proof of Lemma \ref{lem:v13}. Lemma \ref{lem:v12} can be established by employing an analogous argument.

\begin{proof}[Proof of Lemma \ref{lem:v13}] First of all, from the expressions \eqref{eq:v1}, \eqref{eq:v3} and the estimate \eqref{ineq:V_2>y^2:R_3}, we see that 
\begin{align*}
V_{13}=\phi V_1+(1-\phi)V_3 \ge \frac{1}{2}y^2.
\end{align*}
It follows that $V_{13}(x,y)\to\infty$ whenever $|(x,y)|\to\infty$ in $\R_1\cap\R_3$. This verifies the first condition of Definition \ref{def:lyapunov}.

Turning to the Lyapunov property, we observe that $\lambda(x,y)$ defined in \eqref{eq:lambda} satisfies 
\begin{align}\label{obs:bdd lambda}
|\lambda(x,y)|\le 4,\quad (x,y)\in \R_1\cap\R_3.
\end{align}
Also, there exists a constant $\rho>1$ depending only on $\phi$ as in \eqref{form:phi} such that
\begin{align} \label{obs:bdd phis}
\max_{t\in\rbb}\big\{|\phi(t)|,|\phi'(t)|,|\phi''(t)\big\}<\rho.
\end{align}
Recalling $\mathcal{L}$ and $w$ as in \eqref{form:L}-\eqref{eq:w}, we have
\begin{equation}\label{eq:Lv13}
     \mathcal{L}V_{13}=(w-w^q)nx\frac{\partial V_{13}}{\partial x} +\frac{\epsilon_x^2}{2} \frac{\partial^2V_{13}}{\partial x^2} + (-w-w^q)my\frac{\partial V_{13}}{\partial y} + \frac{\epsilon_y^2}{2} \frac{\partial^2V_{13}}{\partial y^2}.
\end{equation}

With regard to $ \frac{\partial^2 V_{13}}{\partial y^2}$, recalling that
\begin{align*}
V_1-V_3=x^2\left(1+k_3y^2\right),
\end{align*}
we may recast the first three terms of $\frac{\partial^2 V_{13}}{\partial y^2}$ on the right-hand side of \eqref{eq:dy2} as
\begin{align}\label{eq:first three terms in partial y-sqr}
&\phi^{''}\cdot\lambda^2\cdot \frac{2^2(n-1)^2}{y^2}\cdot (V_1-V_3)
    +\phi'\cdot \lambda \cdot\frac{2^2(n-1)^2}{y^2} (V_1-V_3) 
    +\phi'\cdot \lambda \cdot \frac{2(1-n)}{y^2}(V_1-V_3) \notag \\
    &=\phi^{''}\cdot\lambda^2\cdot2^2(n-1)^2\frac{x^2}{y^2}+k_3x^2
    +\phi'\cdot \lambda \cdot2^2(n-1)^2\frac{x^2}{y^2}+k_3x^2
    +\phi'\cdot \lambda \cdot 2(1-n)\frac{x^2}{y^2}+k_3x^2.
\end{align}
Also, since
\begin{align*}
\frac{\partial (V_1-V_3)}{\partial y}=2k_3x^2y,
\end{align*}
the fourth term of $\frac{\partial^2 V_{13}}{\partial y^2}$ on the right-hand side of \eqref{eq:dy2} is rewritten as
 \begin{equation}\label{eq:fourth term in partial y-sqr}
       2\phi'\cdot \lambda\cdot \frac{2(n-1)}{y}\frac{\partial (V_1-V_i)}{\partial y}=2\phi'\cdot \lambda \cdot 2^2(n-1)k_3x^2.
 \end{equation}
Letting $c_3$ and $k_3$ be specified according to Lemma \ref{lem:v3}, we note that
\begin{align*}
|y|\ge c_3,\quad x^2\le \frac{1}{k_3},\quad (x,y)\in \mathcal{R}_3.
\end{align*}
Together with observations \eqref{obs:bdd lambda}-\eqref{obs:bdd phis} as well as expressions \eqref{eq:first three terms in partial y-sqr}-\eqref{eq:fourth term in partial y-sqr}, we infer the existence of a positive constant $C=C(n,c_3,k_3)$ such that for $(x,y)\in\R_1\cap \R_3$,
\begin{align*}
\phi^{''}\cdot\lambda^2\cdot \frac{2^2(n-1)^2}{y^2}\cdot (V_1-V_3)
    +\phi'\cdot \lambda \cdot\frac{2^2(n-1)^2}{y^2} (V_1-V_3) 
    +\phi'\cdot \lambda \cdot \frac{2(1-n)}{y^2}(V_1-V_3)\le C,
\end{align*}
and that
\begin{align*}
2\phi'\cdot \lambda\cdot \frac{2(n-1)}{y}\frac{\partial (V_1-V_i)}{\partial y} \le C.
\end{align*}
In view of \eqref{eq:dy2}, we deduce
\begin{align} \label{eq:bound-v13-dy2}
\frac{\partial^2 V_{13}}{\partial y^2}\le  \phi\frac{\partial^2 V_1}{\partial y^2}+    (1-\phi)\frac{\partial^2 V_3}{\partial y^2}+C,\quad (x,y)\in \mathcal{R}_1\cap\mathcal{R}_3.
\end{align}

Concerning $\frac{\partial^2 V_{13}}{\partial x^2}$ on the right-hand side of \eqref{eq:Lv13}, we combine
\begin{align*}
V_1-V_3=x^2\left(1+k_3y^2\right),\quad\text{and}\quad\frac{\partial (V_1-V_3)}{\partial x}=2x\left(1+k_3y^2\right),
\end{align*}
with \eqref{eq:dx2} to obtain the identity
\begin{align}\label{ub:partial x-sqr 1}
\frac{\partial^2 V_{13}}{\partial x^2}
           &=\phi^{''}\cdot\lambda^2\cdot2^2(m-1)^2\left(1+k_3y^2\right)
             +
             \phi'\cdot\lambda \cdot 2^2(m-1)^2\left(1+k_3y^2\right) \nt \\
            &\qquad+
            \phi'\cdot \lambda \cdot  2(1-m)\left(1+k_3y^2\right)
             +
             2\phi'\cdot \lambda\cdot 2(m-1)\left(1+k_3y^2\right) \nt \\
            &\qquad+
             \phi\cdot\frac{\partial^2 V_1}{\partial x^2}
             +
            (1-\phi)\cdot\frac{\partial^2 V_3}{\partial x^2}.
\end{align}
Applying \eqref{obs:bdd lambda}-\eqref{obs:bdd phis}, namely, $|\lambda|\le 4$ and $|\phi'|, |\phi''|\le \rho$, we deduce the bound 
        \begin{gather}
            \frac{\partial^2 V_{13}}{\partial x^2}\le C+ 160\rho m^2k_3y^2+\phi\frac{\partial^2 V_1}{\partial x^2}+(1-\phi)\frac{\partial^2 V_3}{\partial x^2},\quad(x,y)\in \mathcal{R}_1\cap\mathcal{R}_3,\label{eq:bound-v13-dx2}
\end{gather}
for some positive constant $C=C(m,\phi)$.

Turning to $\mathcal{L}V_{13}$ as in \eqref{eq:Lv13}, from \eqref{eq:bound-v13-dy2} and \eqref{eq:bound-v13-dx2} together with the expressions \eqref{eq:dx}-\eqref{eq:dy}, we infer
  \begin{align*}
       \mathcal{L}V_{13}&\le \phi\cdot \mathcal{L}V_1 +(1-\phi)\cdot \mathcal{L}V_3+n\phi'\cdot 2(m-1)(w-w^q)\cdot \lambda\cdot (V_1-V_3)\notag\\
       &\qquad+m\phi'\cdot 2(n-1)(-w-w^q)\cdot \lambda\cdot (V_1-V_3) \\
  &\qquad     +C+\frac{\epsilon_x^2}{2}\cdot 160\rho m^2k_3\cdot y^2.       
       \end{align*}
Substituting $V_1-V_3=x^2\left(1+k_3y^2\right)$ into the above right-hand side yields
       \begin{align*}   
     \mathcal{L}V_{13}  &\le \phi\cdot \mathcal{L}V_1 +(1-\phi)\cdot \mathcal{L}V_3+80\epsilon_x^2 \rho m ^2k_3\cdot y^2+C\notag\\
       &\qquad+2w^q\phi'\lambda(m+n-2mn)\cdot x^2\left(1+k_3y^2\right)\notag\\
       &\qquad+2w\phi'\lambda (m-n)\cdot x^2\left(1+k_3y^2\right),\quad(x,y)\in \mathcal{R}_1\cap\mathcal{R}_3.
       \end{align*} 
Recalling the notation $u=|w|$ as in \eqref{eq:u}, from \eqref{obs:bdd lambda}-\eqref{obs:bdd phis} together with the fact that $m,n>1$, cf. Assumption \ref{cond:m_and_n}, we obtain the bound
       \begin{align}\label{ub:lv13.2}
          \mathcal{L}V_{13}&\le \phi\cdot \mathcal{L}V_1 +(1-\phi)\cdot \mathcal{L}V_3+80\epsilon_x^2 \rho m ^2k_3\cdot y^2+C\notag\\
       &\qquad+16\rho mn(u^q+u)\cdot x^2\left(1+k_3y^2\right),\quad  (x,y)\in \mathcal{R}_1\cap \mathcal{R}_3.
       \end{align}
Letting $b$ be the constant as in \eqref{ub:c3} satisfying
       \begin{equation}\label{lem4.2: upper bound of b}
            b>64\rho mn>4,
       \end{equation}       
from the estimate \eqref{ineq:x^2<C_3}, we see that for $x\in\R_3$,
       \begin{align}\label{ub:ub x-sqr for v13}
         x^2\le \frac{m}{b(m+n)k_3}<\frac{1}{bk_3}<\frac{1}{64\rho mn k_3}.
       \end{align} 
Applying \eqref{ub:ub x-sqr for v13} to \eqref{ub:lv13.2} produces 
       \begin{align}\label{ub:lv13.4}
          \mathcal{L}V_{13} &\le \phi\cdot \mathcal{L}V_1 +(1-\phi)\cdot \mathcal{L}V_3+80\epsilon_x^2 \rho m ^2k_3\cdot y ^2+C\notag\\
           &\qquad+16\rho mn (u^q +u)\cdot\frac{1}{64\rho mn k_3}\left(1+k_3y^2\right)\notag\\
           &=\phi\cdot \mathcal{L}V_1 +(1-\phi)\cdot \mathcal{L}V_3+\left(80\epsilon_x^2 \rho m ^2k_3+ \frac{1}{4}u^q+\frac{1}{4}u\right)\cdot y ^2\notag\\
           &\qquad+C+\frac{1}{4k_3}\left(u^q+u\right),\quad (x,y)\in \mathcal{R}_1\cap \mathcal{R}_3.
       \end{align}

Next, to estimate the right-hand side of \eqref{ub:lv13.4}, we invoke Lemma \ref{lem:v1} and Lemma \ref{lem:v3} and obtain
\begin{align*}
&\phi\cdot \mathcal{L}V_1 +(1-\phi)\cdot \mathcal{L}V_3\\
&\le \phi\cdot\left(-\frac{1}{2}V_1-u^qV_1+\epsilon_x^2+\epsilon_y^2\right)+(1-\phi)\cdot \left(-V_3-\frac{1}{2}y^2u^q+\epsilon_y^2\right)\\
&\le -\frac{1}{2}V_{13}-\frac{1}{2}u^q y^2+C.
\end{align*}
Setting 
\begin{align*}
b_{13}= 80\epsilon_x^2 \rho m^2k_3>1,
\end{align*}
it follows from \eqref{ub:lv13.4} that
\begin{align} \label{ub:lv13.5}
         \mathcal{L}V_{13} & \le
          -\frac{1}{2} V_{13}+b_{13}y^2-\frac{1}{4}u^qy^2+\frac{1}{4}uy^2+\frac{1}{4k_3}u^q+\frac{1}{4k_3}u+C,
    \end{align}
    where $C$ is a positive constant independent of $(x,y)$. Recall from Lemma \ref{lem:v1} that $c_1>\frac{4^{\frac{1}{q-1}}}{a}$, and thus \eqref{ub:u>4^(1/(q-1))_in_R1} and \eqref{eq:bound on u} hold in $\R_1$. In particular, 
 \begin{align*}
 u<\frac{1}{2}u^{q}.
\end{align*}  
As a consequence, \eqref{ub:lv13.5} implies the bound
\begin{align}\label{ub:lv13.6}
\mathcal{L}V_{13}  &\le -\frac{1}{2}V_{13}+b_{13}y^2-\frac{1}{4}u^2y^2+\frac{1}{8}u^qy^2+\frac{1}{4k_3}u^q+\frac{1}{8k_3}u^q+C \notag\\
            &=-\frac{1}{2}V_{13}+\left(b_{13}-\frac{1}{16}u^q\right)y^2+\left(\frac{3}{8k_3}-\frac{1}{16}y^2\right)u^q+C.
\end{align}
We emphasize that at this point, other than the condition $c_1>\frac{4^{\frac{1}{q-1}}}{a}$ as in Lemma \ref{lem:v1}, we have not chosen $c_1$ carefully. In what follows, we will pick $c_1$ sufficiently large so as to produce 
\begin{align*}
b_{13}-\frac{1}{16}u^q<0,\quad \text{and}\quad \frac{3}{8k_3}-\frac{1}{16}y^2<0.
\end{align*}
On the one hand, from \eqref{ub:u>4^(1/(q-1))_in_R1}, we see that
\begin{align*}
\frac{1}{16}u^2>\frac{1}{16}a^2c_1^2.
\end{align*}
Thus, provided
 \begin{equation*} 
        c_1>\frac{4}{a}\sqrt{b_{13}},
    \end{equation*}
we immediately obtain
\begin{align*}
b_{13}-\frac{1}{16}u^2<0.
\end{align*}
On the other hand, recall from Lemma \ref{lem:v3} that in $\mathcal{R}_3$,
\begin{align*}
|y|\ge c_3=\left(2c_1\cdot\left(\frac{n}{bk_3(m+n)}\right)^\frac{1-n}{2}\right)^\frac{1}{m-1}.
\end{align*}
In the above, $k_3$ and $b$ are as in \eqref{eq:k3} and \eqref{lem4.2: upper bound of b}, respectively. Pick $c_1$ sufficiently large such that
     \begin{equation*} 
        c_1>\frac{1}{2}\left(\frac{6}{k_3}\right)^\frac{m-1}{2}\left(\frac{n}{bk_3(m+n)}\right)^{\frac{n-1}{2}}.
    \end{equation*}
A routine calculation shows that
    \begin{align*}
    \frac{3}{8k_3}-\frac{1}{16}y^2<0.
    \end{align*}
Altogether, choosing $c_1$ sufficiently large such that
\begin{equation*}  
        c_1>\max\left\{\frac{1}{2}\left(\frac{6}{k_3}\right)^\frac{m-1}{2}\left(\frac{n}{bk_3(m+n)}\right)^{\frac{n-1}{2}},\frac{4}{a}\sqrt{ 80\epsilon_x^2 \rho m^2k_3},\frac{4^{\frac{1}{q-1}}}{a}\right\}, 
    \end{equation*}
from \eqref{ub:lv13.6}, we arrive at the Lyapunov bound
\begin{align*}
\mathcal{L}V_{13}  &\le -\frac{1}{2}V_{13}+C,\quad (x,y)\in\R_1\cap\R_3,
\end{align*}
for some positive constant $C$ independent of $(x,y)$. This finishes the proof.

\end{proof}

\section{Proof of Theorem \ref{thm:ergodicity}} \label{sec:proof-of-main-theorem}

In this section, we provide the proof of Theorem \ref{thm:ergodicity}, whose argument makes use of the Lyapunov construction in Section \ref{sec:lyapunov} as well as a minorization condition. For the reader's convenience, we recall the definition of the latter below.

\begin{definition} \label{def:minorization}
 Denote 
\begin{align*}
B(\x,R) =\big\{\y\in\rbb^2:|\x-\y|\le R\big\}.
\end{align*}
The system \eqref{eq:system} is said to satisfy a \emph{minorization} condition if for all $R$ sufficiently large, there exist positive constants $t_R,\,\gamma_R$ and a probability measure $\nu_R$ on $\rbb^2$ such that for every $\x\in B((0,0),R)$ and any Borel set $A\subset \rbb^2$,
\begin{align} \label{ineq:minorization}
P_{t_R}\big(\x,A\big)\ge \gamma_R\nu_R(A).
\end{align}
\end{definition}

The minorization condition as in Definition \ref{def:minorization} is summarized in the following auxiliary result whose proof is relatively standard following the classical control theory of SDEs \cite{mattingly2002ergodicity,stroock1972degenerate,
stroock1972support}.

\begin{lemma} \label{lem:minorization}
The system \eqref{eq:system} satisfies the minorization condition as in Definition \ref{def:minorization}.
\end{lemma}

For the sake of clarity, the proof of Lemma \ref{lem:minorization} will be deferred to the end of this section. Nevertheless, assuming Lemma \ref{lem:minorization}, we are now in a position to conclude the proof of Theorem \ref{thm:ergodicity} by verifying the conditions of \cite[Theorem 1.2]{hairer2011yet}, which we recall below for the sake of completeness.

\begin{theorem}{\cite[Theorem 1.2]{hairer2011yet}}
    Given a measurable space $\mathbf{X}$ and a Markov transition kernel $P$ on $\Xbf$, suppose that the followings hold:
    
    1. \textup{(\cite[Assumption 1]{hairer2011yet}) }There exists a function $V:\Xbf\to[0,\infty)$ and constants $c\in (0,1)$ and $K\ge 0$ such that for all $\x\in \Xbf$
    \begin{align*}
        P V (\x) \le cV(\x)+K,
    \end{align*}
    where $PV(\x) = \int_{\Xbf}V(\y)P(\x,\d\y)$.

    2. \textup{(\cite[Assumption 2]{hairer2011yet})} There exist a positive constant $\alpha\in(0,1)$ and a probability measure $\nu$ such that 
    \begin{align*}
        \inf_{\x\in \mathcal{C}} P(\x,\cdot) \le \alpha\nu(\cdot),
    \end{align*}
    where $\mathcal{C}=\{\y\in \Xbf:V(\y)\le R\}$ for some $R\ge 2K/(1-c)$ where $c$ and $K$ are the constants from condition 1.

    Then, $P$ admits a unique invariant probability measure $\nu$. Moreover, there exist positive constant $\gamma\in(0,1)$ and $C>0$ such that for all probability measure $\mu$ on $\Xbf$,
    \begin{align*}
        \W_V(P_n^*\mu,\nu) \le C\gamma^n\W_V(\mu,\nu).
    \end{align*}

\end{theorem}

\begin{proof}[Proof of Theorem \ref{thm:ergodicity}] On the one hand, from Lemma \ref{lem:Lyapunov}, $V$ constructed in Section \ref{sec:lyapunov} is a global Lyapunov function for \eqref{eq:system}. In particular, this satisfies \cite[Assumption 1]{hairer2011yet}. On the other hand, the minorization condition established in Lemma \ref{lem:minorization} verifies \cite[Assumption 2]{hairer2011yet}. In view of \cite[Theorem 1.2]{hairer2011yet}, we obtain a unique invariant probability measure $\pi$ for \eqref{eq:system} as well as the exponential convergent rate \eqref{ineq:exponential-rate}, as claimed.

\end{proof}

Turning back to Lemma \ref{lem:minorization}, in order to establish the minorization condition, we will make use of the Stroock-Varadhan Support Theorem \cite{stroock1972degenerate,stroock1972support} as well as a control argument \cite{mattingly2002ergodicity} showing that the dynamics can always be driven toward the center of the phase space. Together with the ellipticity \cite{hormander1967hypoelliptic}, they will allow us to obtain the desired property. Ultimately, this result is combined with the Lyapunov function to conclude the exponential convergent rate in Theorem \ref{thm:ergodicity}.

\begin{proof}[Proof of Lemma \ref{lem:minorization}] 
Denote 
\begin{align*}
X_1=\epsilon_x\partial
_x,\quad \text{and}\quad X_2=\epsilon_y\partial_y.
\end{align*}
Observe that for every $\x\in\rbb^2$, $\{X_1(\x),X_2(\x)\}$ spans $\rbb^2$. In light of \cite[Corollary 7.2]{bellet2006ergodic}, the Markov transition probabilities $P_t(\x,\cdot)$ admits a smooth probability density $p(t,\x,\y):(0,\infty)\times\rbb^2\times\rbb^2\to [0,\infty)$. Furthermore, the Markov semigroup $\P_t(\x,\cdot)$ is strong Feller, i.e., for all $\f \in \B_b(\rbb^2)$, $P_t\f(\x)\in C^1(\rbb^2)$. In particular, for all $R>0$ and Borel set $A$, $P_t(\x,A)=P_t(1_A)(\x)$ is a continuous function with respect to $\x$. 

\newcommand{\xt}{\tilde{x}}
\newcommand{\yt}{\tilde{y}}

Next, consider the control problem
\begin{align}\label{eq:system:control}
        \d{}\xt_t& = \Big(h'(\xt^m_t \yt^n_t)\xt^{m-1}_t \yt^{n-1}_t -\big|h'(\xt^m_t \yt^n_t )\xt^{m-1}_t \yt^{n-1}_t \big|^q\Big)n \xt_t  \d t + \epsilon_x \d{}U_t^1, \nt \\ 
 \d{}\yt_t& = \Big(- h'(\xt^m_t \yt^n_t )\xt^{m-1}_t \yt^{n-1}_t  -\big|h'(\xt^m_t \yt^n_t )\xt^{m-1}_t \yt^{n-1}_t \big|^q\Big) m\yt_t \d t  + \epsilon_y \d{}U_t^2,
\end{align} 
where $(U^1_t,U^2_t)\in C^1(\rbb;\rbb^2)$ is a control process with $(U^1_0,U^2_0)=(0,0)$. Picking the trivial processes $\xt_t=\yt_t=U^1_t=U^2_t=0$, $t\ge 0$, observe that $(\xt_t,\yt_t,U^1_t,U^2_t)$ solves the control problem \eqref{eq:system:control} and drives the origin at time $0$ to the origin at time $t$. In light of the Stroock-Varadhan Support Theorem \cite[Theorem 6.1]{bellet2006ergodic}, \cite{
stroock1972degenerate,
stroock1972support}, we infer a positive constant $R_1>0$ such that $P_t((0,0),B((0,0),R_1))>0$. As a consequence, there exists $\y^*\in B((0,0),R_1)$ satisfying $p(t,(0,0),\y^*)>0$. Together with the smoothness of $p(t,\cdot,\cdot)$, we obtain the following infimum
\begin{align} \label{ineq:control:inf_(x,y)p(t,x,y)>0}
\inf_{\x\in B((0,0),\varepsilon_1),\y\in B(\y^*,\varepsilon_2)}p(t,\x,\y)>0,
\end{align}
for some positive constants $\varepsilon_1,\varepsilon_2$. Also, for any $(x,y)\in B((0,0),R)$, let $(\xt_t,\yt_t)\in C^1(\rbb;\rbb^2)$ be such that $(\xt_0,\yt_0)=(x,y)$ and $(\xt_t,\yt_t)=0$. Consider the control process
\begin{align*}
U^1_t&=\frac{1}{\epsilon_x}\int_0^t \xt_s'- \Big(h'(\xt^m_s \yt^n_s)\xt^{m-1}_s \yt^{n-1}_s -\big|h'(\xt^m_s \yt^n_s )\xt^{m-1}_s \yt^{n-1}_s \big|^q\Big)n \xt_s\d s,\\
U^2_t&=\frac{1}{\epsilon_y}\int_0^t \yt_s'+\Big( h'(\xt^m_s \yt^n_s )\xt^{m-1}_s \yt^{n-1}_s  +\big|h'(\xt^m_s \yt^n_s )\xt^{m-1}_s \yt^{n-1}_s \big|^q\Big) m\yt_s \d s.
\end{align*}
observe that $(\xt_t,\yt_t,U^1_t,U^2_t)$ defined above solves the control problem \eqref{eq:system:control} and drives $(x,y)$ at time $0$ to the origin at time $t$. We invoke the Stroock-Varadhan Theorem again to obtain 
\begin{align*}
P_t((x,y),B((0,0),\varepsilon_1))>0.
\end{align*}
By the strong Feller property, we deduce
\begin{align} \label{ineq:control:strong_Feller}
\inf_{\x\in B((0,0),R)}P_t(\x,B((0,0),\varepsilon_1))>0.
\end{align}

Now, define the following probability measure $\nu$ on $\rbb^2$ by
\begin{align*}
\nu(A)=\frac{|A\cap B(\y^*,\varepsilon_2)|}{|B(\y^*,\varepsilon_2)|},
\end{align*}
where for a slight abuse of notation, $|\cdot|$ denotes Lebesgue measure on $\rbb^2$. For every $\x\in B((0,0),R)$ and Borel set $A$, we have the following chain of estimates while making use of Markov property and \eqref{ineq:control:inf_(x,y)p(t,x,y)>0}-\eqref{ineq:control:strong_Feller}
\begin{align*}
    P_{2t}(\mathrm{x},A) & =\int_A\Big[ \int_{\rbb^2}P_t(\mathrm{x},\d\mathrm{y}) \Big]P_t(\mathrm{y},\d\mathrm{z}) =\int_{\rbb^2} \Big[ \int_AP_t(\mathrm{y},\d\mathrm{z}) \Big]P_t(\mathrm{x},\d\mathrm{y})\\
    &=\int_{\rbb^2} \Big[ \int_A p(t,\mathrm{y},\mathrm{z})\d\mathrm{z} \Big]P_t(\mathrm{x},\d\mathrm{y})\\
    &\ge \int_{B((0,0),\varepsilon_1)}\Big[ \int_{A\cap B(\mathrm{y}^*,\varepsilon_2)}p(t,\mathrm{y},\mathrm{z})\d\mathrm{z} \Big]P_t(\mathrm{x},\d\mathrm{y})\\
    &\ge \inf_{ \mathrm{y}\in B((0,0),\varepsilon_1)
    \atop \mathrm{z} \in  B(\mathrm{y}^*,\varepsilon_2)} p(t,\mathrm{y},\mathrm{z})\Big( \int_{B((0,0),\varepsilon_1)}P_t(\mathrm{x},\d\mathrm{y})\Big)|A\cap B(\y^*,\varepsilon_2)|\\
    &\ge \inf_{ \mathrm{y}\in B((0,0),\varepsilon_1)
    \atop \mathrm{z} \in  B(\mathrm{y}^*,\varepsilon_2)} p(t,\mathrm{y},\mathrm{z})\Big( \inf_{\x\in B((0,0),R)} P_t(\mathrm{x},B((0,0),\varepsilon_1))\Big)|A\cap B(\y^*,\varepsilon_2)|\\
    &\ge \gamma\nu(A),
\end{align*}
where
\begin{align*}
\gamma=\inf_{\x\in B((0,0),R)}P_t(\x,B((0,0),\varepsilon_1)) \cdot |B(\y^*,\varepsilon_2)|\cdot \inf_{\y\in B((0,0),\varepsilon_1),\z\in B(\y^*,\varepsilon_2)}p(t,\y,\z).
\end{align*}
This establishes the minorization property, thereby finishing the proof.
\end{proof}

\section*{Acknowledgments}

The authors would like to thank anonymous referees for their providing a thorough review of this work. We appreciate their careful reading and insightful comments, which have improved the manuscript.

\section*{conflict of interest}
The authors have no conflicts of interest to declare that are relevant to the content of this article.


\bibliographystyle{abbrv}
{\footnotesize\bibliography{Revision_V2}}
%

\end{document}